\documentclass[11pt]{amsart}
\usepackage{amsmath,amssymb,mathrsfs}
\usepackage[colorlinks, linkcolor=red, citecolor=green,filecolor=magenta,urlcolor=cyan]{hyperref}

\allowdisplaybreaks

\newcommand{\abs}[1]{\left\lvert#1\right\rvert}
\newcommand{\kh}[1]{\left(#1\right)}%
\newcommand{\zkh}[1]{\left[#1\right]}%
\newcommand{\dkh}[1]{\left\{#1\right\}}%

\newtheorem{theorem}{Theorem}

\newtheorem{lem}[theorem]{Lemma}
\newtheorem*{conj}{Conjecture}

\begin{document}

\title[Minimal hypersurface in $\mathbb H^5$]{Complete minimal hypersurfaces in $\mathbb H^5$ with constant scalar curvature and zero Gauss-Kronecker curvature}

\author[Qing Cui]{Qing Cui }
\address[Qing Cui]{School of Mathematics,
Southwest Jiaotong University, 611756
Chengdu, Sichuan, China}

\author[Boyuan Zhang]{Boyuan Zhang }
\address[Boyuan Zhang]{School of Mathematics,
Southwest Jiaotong University, 611756
Chengdu, Sichuan, China}




\email{cuiqing@swjtu.edu.cn, \ by$\_$zhang@my.swjtu.edu.cn}

 \begin{abstract}
We show that any complete minimal hypersurface in the five-dimensional hyperbolic space $\mathbb H^5$, endowed with constant scalar curvature and vanishing Gauss-Kronecker curvature, must be totally geodesic.  Cheng-Peng \cite{ChePen24} recently conjecture that  any complete minimal hypersurface with constant scalar curvature in $\mathbb H^4$ is totally geodesic. Our result partially confirms this conjecture in five dimensional setting.
\end{abstract}
\maketitle

\section{Introduction}
Let $\mathbb Q^3(c)$ be the three dimensional space form with constant sectional curvature $c\in \mathbb R$.
For  a complete minimal surface $M$ in $\mathbb Q^3(c)$ with constant Gaussian curvature $K$ and principal curvatures $\lambda_1$ and $\lambda_2$, the minimality condition combined with Gauss equation gives
\begin{align*}
\lambda_1+\lambda_2=0, \quad K=c+\lambda_1\lambda_2.
\end{align*}
This implies that both $\lambda_1$ and $\lambda_2$ are constant, hence $M$ is isoparametric. Furthermore, up to scaling $c$, $M$ must be isometric to one of the following four surfaces:
\begin{itemize}
\item The totally geodesic plane $\mathbb R^2$ in $\mathbb R^3$;
\item The totally geodesic sphere $\mathbb S^2$ in round 3-sphere $\mathbb S^3$;
\item The Clifford torus $\mathbb S^1\kh{\frac{1}{\sqrt{2}}}\times\mathbb S^1\kh{\frac{1}{\sqrt{2}}}$ in $\mathbb S^3$;
\item  The totally geodesic hyperbolic plane $\mathbb H^2$ in three dimensional hyperbolic space $\mathbb H^3$. 
\end{itemize}

For higher dimensions,  mathematicians 
investigate complete minimal hypersurfaces with constant scalar curvature in space forms $\mathbb Q^{n+1}(c)$. For $c>0$, up to scaling $c$, $\mathbb Q^{n+1}(c)$ is isometric to $\mathbb S^{n+1}$. Complete, especially closed, minimal hypersurfaces with constant scalar curvature in $\mathbb S^{n+1}$ have been studied intensively by Simons \cite{Sim68}, Chern-do Carmo-Kobayashi \cite{ChedoCKob70},
 Lawson \cite{Law69}, Peng-Terng \cite{PenTer83} and many other authors. Mathematicians conjecture that:
{\it A closed minimal hypersurface with constant scalar curvature in  $\mathbb S^{n+1}$ must be isoparametric.}
 This conjecture is commonly referred to as {\it the strong version of the Chern conjecture}, which 
 was proved by Chang \cite{Cha93} for $n=3$ and has been extensively investigated by many mathematicians over the past three decades.

For $c\le 0$, up to scaling $c$, $\mathbb Q^{n+1}(c)$ is isometric to either $\mathbb R^{n+1}$ or $\mathbb H^{n+1}$. In these two space forms, unlike the spherical  case, as far as we know, there are no known examples of complete minimal hypersurfaces with constant scalar curvature other than the totally geodesic ones. Can one obtain similar results as in dimension two by replacing the Gaussian curvature with the scalar curvature? In 1994, Cheng-Wan \cite{CheWan94} provided the first  affirmative answer for the case $c=0$ and $n=3$. Specifically, they proved that {\it a complete minimal hypersurface with constant scalar curvature in $\mathbb R^4$ must be totally geodesic}. 
Moreover, they gave a complete classification of constant mean curvature hypersurfaces in $\mathbb R^4$ with constant scalar curvature. 
Recently, Cheng-Peng \cite{ChePen24} studied complete minimal hypersurfaces in the hyperbolic space $\mathbb H^4(-1)$ with constant scalar curvature. They proved that the second fundamental form $A$ of such a hypersurface must satisfy $\abs{A}^2<\dfrac{21}{29}$. Similar to the two-dimensional case, the authors   conjectured the following:
\begin{conj}[\cite{ChePen24}]
A complete minimal hypersurfaces with constant scalar curvature in
the hyperbolic space $\mathbb H^4$ is totally geodesic.
\end{conj}
The results of Cheng-Peng  \cite{ChePen24} strongly support the above conjecture. Note that if the above conjecture holds, it can be viewed as a Bernstein-type theorem in hyperbolic space. It was shown by do Carmo-Lawson \cite{doCLaw83} that, in hyperbolic space, the Bernstein theorem is true for all dimensions, which  contrasts with the Euclidean case (where the Bernstein theorem is valid only for $n\le 7$). A natural question arises: {\it Is the above conjecture still true for higher dimensions ($n\ge 4$)?} 

However, in higher dimensions,  the minimality and the constant scalar curvature conditions seem insufficient or extremely challenging  to ensure that the hypersurface is totally geodesic. Observe that  every totally geodesic hypersurface must have zero Gauss-Kronecker curvature. Therefore, in this paper, we focus on the special case where $n=4$ and the hypersurface has zero Gauss-Kronecker curvature. Under these assumptions, we obtain the following result:
\begin{theorem}\label{mainthm}
Let $M$ be a complete minimal hypersurface in $\mathbb H^5$ with constant scalar curvature and zero Gauss-Kronecker curvature, then $M$ is totally geodesic.
\end{theorem}

This paper is organized as follows. Section 2 introduces the fundamental notations and revisits essential formulas and  results, with particular emphasis on the Omori-Yau maximum principle. Subsequently, Section 3 presents the detailed proof of Theorem \ref{mainthm}.

\vspace{1cm}
\section{Preliminary}
In this section, we establish the necessary notations and present several well-known fundamental formulas and results.

Let $M$ be a minimal hypersurface in $\mathbb H^5$ with second fundamental form $A$. We choose a local orthonormal frame $\dkh{e_i}_{i=1}^5$ in $\mathbb H^5$ such that $e_1, \cdots, e_4$ are tangent to $M$. Let $\dkh{\omega^i}_{i=1}^5$ denote the dual frame of $\dkh{e_i}_{i=1}^5$.  The second fundamental form $A$ can then be expressed as $A=\sum_{i,j=1}^4h_{ij}\omega^i\otimes \omega^j$, where $h_{ij}=h_{ji}$. 
Unless otherwise specified, all summations in this  paper range from 1 to 4.
The minimality of $M$ implies that the mean curvature $H$ vanishes, i.e., $H=\dfrac{1}{4}\sum_i h_{ii}=0$.
We define the following notations:
\begin{align*}
F_{,i}=\nabla_i F, \ \ F_{,ij}=\nabla_j\nabla_i F, \ \ h_{ijk}=\nabla_k h_{ij} \ \text{and} \  h_{ijkl}=\nabla_l\nabla_k h_{ij},
\end{align*}
where $\nabla_j$ is the covariant differentiation operator with respect to $e_j$. 
The Gauss equation,  Codazzi equation and  Ricci identities are given by (cf. \cite{ChedoCKob70}):
\begin{align}\label{Gausseq}
R_{ijkl}=&-\kh{\delta_{ik}\delta_{jl}-\delta_{il}\delta_{jk}}+h_{ik}h_{jl}-h_{il}h_{jk}.\\
h_{ijk}=&h_{ikj}.\label{Codazzieq}\\
h_{ijkl}-h_{ijlk}
=&\sum_m h_{im}R_{mjkl}+\sum_m h_{mj}R_{mikl}.\label{Ricciiden}
\end{align}
From the Gauss equation \eqref{Gausseq}, it follows that the scalar curvature $R_M$ of $M$ is given by:
$
R_M=\sum_{i,j}R_{ijij}=-12 -\abs{A}^2.
$
Consequently, $\abs{A}^2$ is constant if and only if the scalar curvature $R_M$ is constant.

Assume $A$ has principal curvature $\lambda_1,\cdots, \lambda_4$ at a given point. At this point, combining the Gauss equation \eqref{Gausseq} with  the Ricci identity \eqref{Ricciiden} yields
\begin{align}
h_{ijkl}-h_{ijlk}=\kh{\lambda_i-\lambda_j}\kh{-1+\lambda_i\lambda_j}\kh{\delta_{ik}\delta_{jl}-\delta_{il}\delta_{jk}}.\label{tijkl}
\end{align}

For each positive integer $k$, we denote by 
\begin{align}\label{f_k}
f_k={\rm tr} A^k,\quad \sigma_k=\sum_{i_1<\cdots<i_k}\lambda_{i_1}\cdots\lambda_{i_k}.
\end{align}
The functions $f_k$ and $\sigma_k$ are  globally well-defined smooth functions satisfying the following identities:
\begin{align}
\begin{cases}\label{f1234}
&f_1 = \sigma_1=4H=0, \\ 
&f_2 =\sigma_1^2-2\sigma_2=\abs{A}^2,\\
&f_3 = \sigma_1^3-3\sigma_1\sigma_2+3\sigma_3
=3\sigma_3,\\
&f_4 = \sigma_1^4-4\sigma_1^2\sigma_2
+4\sigma_1\sigma_3+2\sigma_2^2-4\sigma_4=\frac{1}{2}\abs{A}^4-4\mathcal{K},
\end{cases}
\end{align}
where $\mathcal{K}=\sigma_4=\lambda_1\lambda_2\lambda_3
\lambda_4$
is the Gauss-Kronecker curvature. Notably, 
\begin{align}\label{K=0}
\mathcal{K}\equiv 0 \ \ \text{if and only if}\ \ f_4\equiv\dfrac{1}{2}\abs{A}^4.
\end{align}
The following identities, due to Peng-Terng \cite{PenTer83}, hold:
\begin{align}
\Delta f_3=&-3\kh{4+\abs{A}^2}f_3+6\mathscr{C}, \label{Deltaf3}\\
\Delta f_4 =& -4\kh{4+\abs{A}^2}f_4 +4\kh{2\mathscr{A}+\mathscr{B}},\label{Deltaf4}
\end{align}
where 
\begin{align}\label{ABC}
\mathscr{A}=\sum_{i,j,k} \lambda_i^2 h_{ijk}^2,\quad \mathscr{B}=\sum_{i,j,k} \lambda_i\lambda_jh_{ijk}^2,\quad \mathscr{C}=\sum_{i,j,k}\lambda_ih_{ijk}^2.
\end{align}

If $M$ has constant scalar curvature, Simons \cite{Sim68} and Peng-Terng \cite{PenTer83} established the following identities:
\begin{align}
\abs{\nabla A}^2=&\sum_{i,j,k} h_{ijk}^2=\abs{A}^2\kh{\abs{A}^2+4 }.\label{nablaA}\\
\abs{\nabla^2 A}^2=&\sum_{i,j,k,l}h_{ijkl}^2=  \abs{A}^2\kh{\abs{A}^2+4}\kh{\abs{A}^2+11}+3\kh{\mathscr{A}-2\mathscr{B}}.
\label{nabla2A}
\end{align}

When $M$ is noncompact, a smooth function may not attain its extremum (maximum or minimum). To address this issue, we invoke the following {\it Omori-Yau Maximum Principle} ( see \cite{Omo67} and \cite{Yau75}).
\begin{lem}\label{Omori-Yaulem}
Let $M $ be a complete Riemannian manifold with sectional  curvature bounded from below (or above).
If a $\mathcal C^2$-function $f$ is bounded from above (or below) on $M $, then there exists a sequence   $\{p_{k}\}_{k=1}^{\infty} \subset{M^{n}}$ such that
\begin{enumerate}
\item $\lim\limits_{k \to \infty} f(p_{k})=\underset{M }{\sup}\ f  \quad (\text{or}\ \lim\limits_{k \to \infty} f(p_{k})=\underset{M }{\inf}\ f)$;
\item $\lim\limits_{k \to \infty}|\nabla f(p_{k})|=0$;
\item $\lim\limits_{k \to \infty}\sup\ \nabla_l\nabla_lf(p_{k})\leq 0  \quad (\text{or}\ \lim\limits_{k \to \infty}\inf\ \nabla_l\nabla_lf(p_{k})\geq 0)$, for  $l=1,  2,  \cdots, n$.
\end{enumerate}
\end{lem}

We also require an elementary algebraic lemma, which can be proven using the  method of  Lagrange  multipliers.
\begin{lem}\label{algebraiclem}
Let $a,b,c\in \mathbb R$ satisfy $a\le b\le c$ and $a+b+c=0$. Then, 
\begin{align*}
-\dfrac{1}{\sqrt{6}}\kh{a^2+b^2+c^2}^{3/2}\le a^3+b^3+c^3\le \dfrac{1}{\sqrt{6}}\kh{a^2+b^2+c^2}^{3/2},
\end{align*}
and the equalities hold as follows:
\begin{align*}
a^3+b^3+c^3=&-\dfrac{1}{\sqrt{6}}\kh{a^2+b^2+c^2}^{3/2} \ \ \text{if and only if}\ \ b=c=-\dfrac{a}{2},\\
a^3+b^3+c^3=& \dfrac{1}{\sqrt{6}}\kh{a^2+b^2+c^2}^{3/2} \ \ \text{if and only if}\ \ a=b=-\dfrac{c}{2},\\
a^3+b^3+c^3=&0 \ \ \text{if and only if} \ \ b=0 \ \ \text{and} \ \ a=-c.
\end{align*}
\end{lem}

\vspace{1cm}
\section{Proof of the main theorem}
The main idea of the proof of the main theorem is as follows: assuming 
$M$ is not totally geodesic, we consider the function $f_3$  and derive a contradiction using Lemma \ref{Omori-Yaulem}. To ensure the validity of this approach, it is essential to verify that all identities involving the principal curvatures, $h_{ijkl}$ and $h_{ijkl}$ remain applicable when passing to a subsequence of the sequence provided by Lemma \ref{Omori-Yaulem}. 

\begin{lem}\label{subsequencelem}
Let $M$ be a complete minimal hypersurface in $\mathbb H^5$ with constant scalar curvature,  then there exists a sequence of $\{p_{k}\}_{k=1}^{\infty} \subset{M }$ such that
\begin{enumerate}
\item $\lim\limits_{k \to \infty} f_3(p_{k})=\underset{M }{\sup}\ f_3  \quad (\   \text{or }\   \lim\limits_{k \to \infty} f_3(p_{k})=\underset{M }{\inf}\ f_3)$;
\item $\lim\limits_{k \to \infty}|\nabla f_3(p_{k})|=0$;
\item $\lim\limits_{k \to \infty}\sup\ \nabla_l\nabla_l f_3(p_{k})\leq 0  \quad (\   \text{or }\  \lim\limits_{k \to \infty}\inf\ \nabla_l\nabla_l f_3(p_{k})\geq 0)$, for  $l=1,  2,  \cdots, n$.
\end{enumerate}
Furthermore, there exists a subsequence $\dkh{p_{k_m}}_{m=1}^\infty$ of  $\dkh{p_k}_{k=1}^\infty$, such that   the limits 
\begin{align*}
\hat{\lambda_i}=\lim_{m\to \infty} \lambda_i\kh{p_{k_m}}, \quad \hat{h_{ijk}}=\lim_{m\to \infty} h_{ijk}\kh{p_{k_m}},\quad \hat{h_{ijkl}}=\lim_{m\to \infty} h_{ijkl}\kh{p_{k_m}}
\end{align*}
exist for all $i,j,k,l$.  Moreover, the limits $\hat{\lambda_i}, \hat{h_{ijk}}, \hat{h_{ijkl}}$ satisfy the same identities listed in Section 2 as $\lambda_i, h_{ijk}, h_{ijkl}$.
\end{lem}
\begin{proof}
Since $M$ has constant scalar curvature, $\abs{A}^2$ is a constant, and the sectional curvature of $M$ is bounded by the Gauss equation \eqref{Gausseq}. By the definition of $f_3$, $\abs{f_3}\le \dfrac{1}{\sqrt{3}}\abs{A}^3$ is also bounded. Consequently, Lemma \ref{Omori-Yaulem} applies to $f_3$, and items (1), (2), (3) hold. Since $\abs{A}^2=\sum_i\lambda_i^2$, $\lambda_1, \cdots, \lambda_n$ are bounded continuous functions. For all $i,j,k$, $h_{ijk}$ are bounded by \eqref{nablaA}, and $\mathscr{A}, \mathscr{B}, \mathscr{C}$ are bounded smooth functions by \eqref{ABC}. It follows from \eqref{nabla2A} that $h_{ijkl}$ are also bounded for all $i,j,k,l$.

For these bounded continuous functions, there exists a subsequence $\dkh{p_{k_m}}_{m=1}^\infty$ of $\subset \dkh{p_k}_{k=1}^\infty$, such that these functions  converge on this subsequence. By continuity, the limit functions satisfy all identities listed in Section 2.
\end{proof}

 Moreover, if $M$ has zero Gauss-Kronecker curvature, we establish the following Lemma.
\begin{lem}\label{lineareqs}
Let $M$ be a complete minimal hypersurface in $\mathbb H^5$ with constant scalar curvature and zero Gauss-Kronecker curvature. Then, the limits $\hat{\lambda_i}, \hat{h_{ijk}}, \hat{h_{ijkl}}$ defined in Lemma \ref{subsequencelem} satisfy the following equations for each $1\le k\le 4$:
\begin{align}
\begin{cases}
&\sum_i \hat{h_{iik}}=0,\\
&\sum_i \hat{\lambda_i} \hat{h_{iik}}=0,\\
&\sum_i \hat{\lambda_i}^2 \hat{h_{iik}}=0,\\
&\sum_i \hat{\lambda_i}^3 \hat{h_{iik}}=0,
\end{cases}  , \label{lineareq1}
\end{align}
and for each $k$ and $l$, 
\begin{align}\label{hijkl}
\begin{cases}
&\sum_i \hat{h_{iikl}}=0,\\
&\sum_{i,j} \hat{h_{ijk}} \hat{h_{ijl}}+\sum_i\hat{\lambda_i} \hat{h_{iikl}}=0,\\
&\sum_{i}\kh{\hat{\lambda_i}^3 \hat{h_{iikl}}}+\sum_{i,j}\kh{2\hat{\lambda_i}^2 \hat{h_{ijk}} \hat{h_{ijl}}+\hat{\lambda_i} \hat{\lambda_j} \hat{h_{ijk}} \hat{h_{ijl}}}=0.
\end{cases}  
\end{align}
\end{lem}
\begin{proof}
Since $H\equiv 0$, $\abs{A}^2$ and $f_4=\dfrac{1}{2}\abs{A}^4$ are constants, the first, second and fourth identities in \eqref{lineareq1} are derived by taking covariant derivative with $e_k$ for these three terms, and then passing to the limit on the subsequence defined in Lemma \ref{subsequencelem}. By item (2) of Lemma \ref{subsequencelem}, $\lim_{j\to \infty} \nabla_k f_3(p_j)=0$. Taking the limit on the subsequence yields the third identity of \eqref{lineareq1}.

Similarly, by taking the second covariant derivative with respect to $e_k$ and $e_l$ for $H, \abs{A}^2, f_4$,  and then passing to the limit on the subsequence, we obtain \eqref{hijkl}.
\end{proof}

In the subsequent discussion, when referring to the limits defined in Lemma \ref{subsequencelem} and Lemma \ref{lineareqs}, we will omit the hat symbol `` $\hat{\ }$ '' for simplicity.

We now proceed to present the proof of the main theorem.

\begin{proof}[Proof of Theorem \ref{mainthm}]
Suppose   $M$ is not totally geodesic. Then, $\abs{A}$ is a positive constant. Note that the Gauss-Kronecker curvature $\mathcal{K}\equiv 0$ implies
$f_4=\dfrac{1}{2}\abs{A}^4$ is also a constant. Consequently, \eqref{Deltaf4} simplifies to
\begin{align}\label{Deltaf4-1}
2\mathscr{A}+\mathscr{B}=\dfrac{1}{2}\kh{4+\abs{A}^2}\abs{A}^4.
\end{align}
The condition $\mathcal{K}\equiv 0$ further implies 
 that at least one of the principal curvatures vanishes at each point. By Lemma \ref{algebraiclem}, we have
\begin{align}\label{infsupf_3}
-\dfrac{1}{\sqrt{6}}\abs{A}^3 \le \inf f_3 \le \sup f_3 \le \dfrac{1}{\sqrt{6}} \abs{A}^3.
\end{align}
Based on this observation, we divide the proof into three cases:
\\
{\it Case} (1): $\sup f_3 =\dfrac{1}{\sqrt{6}}\abs{A}^3$.\\
In this case, by Lemma \ref{algebraiclem} and by taking limits on the subsequence in Lemma \ref{subsequencelem}, the limits of the principal curvatures can be expressed as
\begin{align}
\lambda_1=\lambda_2=-\lambda<0=\lambda_3<\lambda_4=2\lambda>0.\label{lambda1multi2}
\end{align}
Solving the linear equation \eqref{lineareq1}, we obtain
\begin{align}\label{h11kh22k}
h_{11k}+h_{22k}=0, \ \ h_{44k}=h_{33k}=0, \ \ \text{for each}\ k.
\end{align}
Consequently, the squared norm of the covariant derivative of  $A$ and the quantities $\mathscr{A}$ and $\mathscr{B}$ are given by
\begin{align*}
\abs{\nabla A}^2=&4\kh{h_{111}^2+h_{112}^2}+6\kh{h_{113}^2+h_{114}^2+h_{123}^2+h_{124}^2+h_{134}^2+h_{234}^2}.\\
\mathscr{A}=&\sum_{i,j,k}\lambda_i^2 h_{ijk}^2\\
=&\lambda^2\sum_{j,k}\kh{h_{1jk}^2+h_{2jk}^2+4h_{4jk}^2}\\
=&\lambda^2\zkh{4\kh{h_{111}^2+h_{112}^2+h_{113}^2+h_{123}^2}+10\kh{h_{234}^2+h_{134}^2}+12\kh{h_{114}^2+h_{124}^2}}.\\
\mathscr{B}=&\sum_{i,j,k}\lambda_i\lambda_jh_{ijk}^2\\
=&\lambda^2\sum_k\kh{h_{11k}^2+h_{22k}^2+2h_{12k}^2-4h_{14k}^2-4h_{24k}^2}\\
=&\lambda^2\zkh{2\kh{h_{113}^2+h_{123}^2}+4\kh{h_{111}^2+h_{112}^2-h_{134}^2-h_{234}^2}-6\kh{h_{114}^2+h_{124}^2} }.
\end{align*}
Combining these expressions   with  \eqref{nablaA} and  \eqref{Deltaf4-1}, we derive,
\begin{align*}
&3\zkh{4\kh{h_{111}^2+h_{112}^2}+6\kh{h_{113}^2+h_{114}^2+h_{123}^2+h_{124}^2+h_{134}^2+h_{234}^2}}\\
=&
3\abs{\nabla A}^2
=3\abs{A}^2\kh{\abs{A}^2+4}
=\dfrac{1}{\lambda^2}\kh{2\mathscr{A}+\mathscr{B}}\\
=&12\kh{h_{111}^2+h_{112}^2}+10\kh{h_{113}^2+h_{123}^2}+16\kh{h_{234}^2+h_{134}^2}+18\kh{h_{114}^2+h_{124}^2},
\end{align*}
which implies
\begin{align*}
0=8\kh{h_{113}^2+h_{123}^2}+2\kh{h_{134}^2+h_{234}^2}.
\end{align*}
Consequently, we have
\begin{align}\label{h123h234}
h_{113}=h_{123}=h_{134}=h_{234}=0.
\end{align}
Substituting \eqref{lambda1multi2}, \eqref{h11kh22k} and  \eqref{h123h234} into \eqref{hijkl}   for $k= l=3$, the equations  reduce to
\begin{align*}
\begin{cases}
&h_{1133}+h_{2233}+h_{3333}+h_{4433}=0,\\
&-\lambda h_{1133}-\lambda h_{2233}+2\lambda h_{4433}=0,\\
&-\lambda^3 h_{1133}-\lambda^3 h_{2233}+8\lambda^3 h_{4433}=0,
\end{cases}
 \end{align*}
which implies $h_{4433}=0$.
On the other hand, for $k=l=4$, \eqref{hijkl} becomes 
\begin{align*}
&\begin{cases}
&h_{1144}+h_{2244}+h_{3344}+h_{4444}=0,\\
&-\lambda h_{1144}-\lambda h_{2244}+2\lambda h_{4444}=-2\kh{h_{114}^2+h_{124}^2},\\
&-\lambda^3 h_{1144}-\lambda^3 h_{2244}+8\lambda^3 h_{4444}=-6\lambda^2 \kh{h_{114}^2+h_{124}^2},
\end{cases}
\end{align*}
which implies $h_{3344}=0$. Consequently, 
\begin{align*}
h_{3344}-h_{4433}=0.
\end{align*}
However, from \eqref{tijkl} and \eqref{lambda1multi2}, we have 
\begin{align*}
h_{3344}-h_{4433}=-2\lambda<0.
\end{align*}
This leads to a contradiction.
\\
{\it Case} (2): $\sup f_3 =-\dfrac{1}{\sqrt{6}} \abs{A}^3$. \\
 In this case, by the inequality \eqref{infsupf_3}, we conclude that $f_3\equiv -\dfrac{1}{\sqrt{6}}\abs{A}^3$ is a constant. Consequently, $f_1=4H=0$, $f_2=\abs{A}^2$, $f_3$ and $f_4=\dfrac{1}{2}\abs{A}^4$ are all constants. This implies that the principal curvatures $\lambda_1, \lambda_2, \lambda_3, \lambda_4$ are constants, and thus $M$ is isoparametric. However, any isoparametric minimal hypersurface in $\mathbb H^5$ must be totally geodesic (cf. \cite[p. 86]{BerConOlm03}), which contradicts our assumption that $M$ is not totally geodesic.\\
{\it Case }(3): $-\dfrac{1}{\sqrt{6}}\abs{A}^3 < \sup f_3 < \dfrac{1}{\sqrt{6}}\abs{A}^3$.\\ To derive a contradiction, we analyze this case by dividing it into three subcases:
\begin{itemize}
\item[(3-i)]: $-\dfrac{1}{\sqrt{6}}\abs{A}^3 < \sup f_3 <0$. In this subcase, $f_3<0$. Consequently, by \eqref{f1234},  Lemma \ref{algebraiclem} and Lemma \ref{subsequencelem}, the limits of the principal curvatures can be set as
\begin{align*}
\lambda_1<\lambda_2=0<\lambda_3<\lambda_4.
\end{align*} 
Solving the linear equation \eqref{lineareq1}, we find
\begin{align*}
h_{iik}=0, \ \ \text{for each}\ i, k.
\end{align*}
Thus, 
\begin{align*}
\mathscr{C}=&\sum_{i,j,k} \lambda_i h_{ijk}^2
=\dfrac{1}{3}\sum_{i,j,k}\kh{\lambda_i+\lambda_j+\lambda_j} h_{ijk}^2\\
=& \dfrac{1}{3} \kh{\lambda_1+\lambda_3+\lambda_4} h_{134}^2=-\dfrac{1}{3}\lambda_2h_{134}^2=0.
\end{align*}
By Lamma \ref{subsequencelem} item (3), \eqref{Deltaf3} becomes
\begin{align*}
0\ge \Delta f_3=-3\kh{\abs{A}^2+4}f_3+\mathscr{C}=-3\kh{\abs{A}^2+4}f_3>0.
\end{align*}
This leads to a contradiction.

\item[(3-ii)]: $ 0 < \sup f_3 < \dfrac{1}{\sqrt{6}}\abs{A}^3$. In this subcase, after taking the limit on the subsequence in Lemma \ref{subsequencelem}, we have $f_3>0$. Therefore, by \eqref{f1234},  Lemma \ref{algebraiclem} and Lemma \ref{subsequencelem}, the limits of the principal curvatures can be expressed as
\begin{align}\label{4principalcurvature}
\lambda_1=-\lambda<\lambda_2=-\mu<\lambda_3=0<\lambda_4=\lambda+\mu.
\end{align} 
Solving the linear equation \eqref{lineareq1}, we also obtain
\begin{align*}
h_{iik}=0, \ \ \text{for each}\ i, k.
\end{align*}
In this subcase, 
\begin{align*}
\abs{\nabla A}^2=6\kh{h_{123}^2+h_{124}^2+h_{134}^2+h_{234}^2}.
\end{align*}
Therefore, from \eqref{ABC} and \eqref{nablaA}, we derive,
\begin{align*}
\mathscr{A}=&\sum_{i,j,k}\lambda_i^2h_{ijk}^2
=\dfrac{1}{3}\sum_{i,j,k}\kh{\lambda_i^2+\lambda_j^2+\lambda_k^2}h_{ijk}^2\\
=&2\kh{h_{123}^2\kh{\abs{A}^2-\lambda_4^2}+h_{124}^2\kh{\abs{A}^2-\lambda_3^2}}\\
&+2\kh{h_{134}^2\kh{\abs{A}^2-\lambda_2^2}+h_{234}^2\kh{\abs{A}^2-\lambda_1^2}}\\
=&\dfrac{1}{3}\abs{A}^4\kh{\abs{A}^2+4}-2C,
\end{align*}
where 
\begin{align*}
C=\lambda_1^2h_{234}^2+\lambda_2^2h_{134}^2 +\lambda_3^2h_{124}^2+\lambda_4^2h_{123}^2.
\end{align*}
\begin{align*}
\mathscr{B}=&\sum_{i,j,k}\lambda_i\lambda_jh_{ijk}^2
=\dfrac{1}{3}\sum_{i,j,k}\kh{\lambda_i\lambda_j+\lambda_j\lambda_k+\lambda_k\lambda_i}h_{ijk}^2\\
=&\dfrac{1}{3}\sum_{i,j,k}\kh{\dfrac{1}{2}\kh{\lambda_i+\lambda_j+\lambda_k}^2-\dfrac{1}{2}\kh{\lambda_i^2+\lambda_j^2+\lambda_k^2}}h_{ijk}^2\\
=&-\dfrac{1}{6}\abs{A}^4\kh{\abs{A}^2+4}+2C.
\end{align*}
Combining the above two formulas with \eqref{Deltaf4-1}, we obtain
\begin{align*}
0=- \dfrac{1}{2}\kh{4+\abs{A}^2}\abs{A}^4 +2\mathscr{A}+\mathscr{B}
=-2C.
\end{align*}
Consequently,
\begin{align*}
h_{123}=h_{134}=h_{234}=0.
\end{align*}
Thus, we have $\abs{\nabla A}^2 =6h_{124}^2$. Then,
taking $k=l=1,2,4$ respectively in \eqref{hijkl}, a direct computation yields,
\begin{align*}
\begin{cases}
&h_{11kk}+h_{22kk}+h_{33kk}+h_{44kk}=0,\\
&-\lambda h_{11kk}-\mu h_{22kk}+\kh{\lambda+\mu} h_{44kk}=-\dfrac{1}{3}\abs{\nabla A}^2,\\
&-\lambda^3 h_{11kk}-\mu^3 h_{22kk}+\kh{\lambda+\mu}^3 h_{44kk}=-\dfrac{1}{6}\abs{A}^2\abs{\nabla A}^2.
\end{cases}
 \end{align*}
 Solving the above equations, we obtain 
 \begin{align*}
h_{33kk}=0, \ \ \text{for} \ \ k=1,2,4.
\end{align*}
Taking $k=l=3$ in \eqref{hijkl}, we obtain 
\begin{align*}
\begin{cases}
&h_{1133}+h_{2233}+h_{3333}+h_{4433}=0,\\
&-\lambda h_{1133}-\mu h_{2233}+\kh{\lambda+\mu} h_{4433}=0,\\
&-\lambda^3 h_{1133}-\mu^3 h_{2233}+\kh{\lambda+\mu}^3 h_{4433}=0,
\end{cases}
 \end{align*}
Solving the above equations, we get
\begin{align}\label{hii33}
h_{1133}=\dfrac{\lambda+2\mu}{\mu-\lambda}h_{4433}.
\end{align}
On the other hand, by \eqref{tijkl}, we obtain
\begin{align*}
h_{1133}=&h_{3311}+\kh{\lambda_1-\lambda_3}\kh{-1+\lambda_1\lambda_3}=\lambda.\\
h_{4433}=&h_{3344}+\kh{\lambda_4-\lambda_3}\kh{-1+\lambda_4\lambda_3}=-\kh{\lambda+\mu}.
\end{align*}
Substituting the above equalities into  \eqref{hii33}, we obtain,
\begin{align*}
\lambda=-\dfrac{\lambda+2\mu}{\mu-\lambda}(\lambda+\mu).
\end{align*}
We finally arrive at $\mu^2+2\lambda\mu=0$, which contradicts with $\lambda>\mu>0$.
\item[(3-iii)]: $\sup f_3 =0.$  In this subcase, to derive a contradiction, we further analyze the following three possibilities separately.
\begin{itemize}
\item[(3-iii-a)]: $\inf f_3=0.$ In this subcase, we have $f_3\equiv 0$. Similar to Case (2), $M$ is isoparametric,  leading to a contradiction.
\item[(3-iii-b)]: $-\dfrac{1}{\sqrt{6}}\abs{A}^3<\inf f_3<0$. 
In this subcase, after taking the limit on the subsequence in Lemma \ref{subsequencelem}, we have $f_3<0$. Therefore, by \eqref{f1234},  Lemma \ref{algebraiclem} and Lemma \ref{subsequencelem}, the limits of the principal curvatures can be expressed as
\begin{align*}
\lambda_1<\lambda_2=0<\lambda_3<\lambda_4,
\end{align*} 
i.e., the limits of the principal curvatures are four distinct numbers. 
Solve the linear equation \eqref{lineareq1}, we also obtain
\begin{align*}
h_{iik}=0, \ \ \text{for each}\ i, k.
\end{align*}
The remaining argument is similar to Case (3-ii).
\item[(3-iii-c)]: $\inf f_3 =-\dfrac{1}{\sqrt{6}}\abs{A}^3$. In this subcase, by Lemma \ref{algebraiclem} and by taking limits on the subsequence in Lemma \ref{subsequencelem}, the limits of the principal curvatures can be set as
\begin{align*}
\lambda_1=-2\lambda<0=\lambda_2<\lambda_3=\lambda_4=\lambda.
\end{align*}
Solve the linear equation \eqref{lineareq1}, we have
\begin{align*}
h_{11k}=h_{22k}=0, \ \ h_{33k}+h_{44k}=0, \ \ \text{for each}\ k.
\end{align*}
The remaining proof is similar to Case (1).
\end{itemize}
\end{itemize}

This completes the proof.
\end{proof}

\vspace{1cm}

\providecommand{\bysame}{\leavevmode\hbox to3em{\hrulefill}\thinspace}
\providecommand{\MR}{\relax\ifhmode\unskip\space\fi MR }


\end{document}